\documentclass[a4paper,12pt]{article}
\usepackage{mathrsfs}
\usepackage{}

\usepackage{amsfonts}
\usepackage{amscd,color}
\usepackage{amsmath,amsfonts,amssymb,amscd}
\usepackage{indentfirst,graphicx,epsfig}
\usepackage{graphicx}
\input{epsf}
\usepackage{graphicx}
\usepackage{epstopdf}
\usepackage{caption}

\newtheorem{thm}{Theorem}[section]

\newtheorem{conj}[thm]{Conjecture}
\newtheorem{lem}[thm]{Lemma}

\newtheorem{cor}[thm]{Corollary}
\newtheorem{pro}[thm]{Proposition}
\newenvironment {proof} {\noindent{\em Proof.}}{\hspace*{\fill}$\Box$\par\vspace{4mm}}
\newcommand{\ml}{l\kern-0.55mm\char39\kern-0.3mm}

\def\qed{\hfill \nopagebreak\rule{5pt}{8pt}}

\title{\textbf{Bounds for the rainbow disconnection number of graphs\footnote{Supported by NSFC No.11871034 and 11531011.}}}
\author{{\small Xuqing Bai, Zhong Huang, Xueliang Li } \\
{\small  Center for Combinatorics and LPMC}\\
{\small Nankai University, Tianjin 300071, China}\\
{\small Email: baixuqing0@163.com, 2120150001@mail.nankai.edu.cn, lxl@nankai.edu.cn}\\
}
\date{}
\begin{document}
\maketitle
\begin{abstract}
An edge-cut $R$ of an edge-colored connected graph is called a rainbow-cut if no two edges in the edge-cut are colored the same. An edge-colored graph is rainbow disconnected if for any two distinct vertices $u$ and $v$ of the graph, there exists a $u$-$v$-rainbow-cut separating them. For a connected graph $G$, the rainbow
disconnection number of $G$, denoted by rd$(G)$, is defined as the
smallest number of colors that are needed in order to make $G$
rainbow disconnected.

In this paper, we first give some tight upper bounds for rd$(G)$, and moreover, we completely characterize the graphs which meet the upper bound of the  Nordhaus-Gaddum type results obtained early by us. Secondly, we propose  a conjecture that $\lambda^+(G)\leq \textnormal{rd}(G)\leq \lambda^+(G)+1$, where $\lambda^+(G)$ is the upper edge-connectivity, and prove the conjecture for many classes of graphs, to support it. Finally, we give the relationship between rd$(G)$ of a graph $G$ and the rainbow vertex-disconnection number rvd$(L(G))$ of the line graph $L(G)$ of $G$.


\end{abstract}

\section{Introduction}

All graphs considered in this paper are finite and
undirected, and all graphs are simple unless emphasized. Let $G=(V(G), E(G))$ be a nontrivial connected graph with vertex-set $V(G)$ and edge-set $E(G)$. For $v\in V(G)$, let $d_G(v)$ and $N_G(v) \ (N_G[v])$  denote the $degree$ and the \emph{open (closed) neighborhood} of $v$ in $G$ (or simply $d(v)$ and $N(v) \ (N[v])$ respectively, when the graph $G$ is clear from the context). We use $\delta(G)$ and $\Delta(G)$ to denote the minimum and maximum degree of $G$, respectively. The notion $G[S]$ denotes the induced subgraph of $G$ by vertex-set $S$.
For any notation or terminology not defined here, we follow those used in \cite{BM}.

Let $G$ be a graph with an \emph{edge-coloring} $c$: $E(G)\rightarrow [k]$, $k \in \mathbb{N}$, where adjacent edges may be colored the same. When adjacent edges of $G$ receive different colors by $c$, the edge-coloring $c$ is called \emph{proper}. The \emph{chromatic index} of $G$, denoted by $\chi'(G)$, is the minimum number of colors needed in a proper edge-coloring of $G$. By a famous theorem of Vizing \cite{V}, one has that
$$\Delta(G) \leq \chi'(G) \leq \Delta(G)+1$$
for every nonempty graph $G$. If $\chi'(G) = \Delta(G)$, then $G$ is said to be in \emph{Class} $1$; if $\chi'(G) = \Delta(G) + 1$, then $G$ is said to be in \emph{Class} $2$.

A path is called \emph{rainbow} if no two edges of the path are colored the same. An edge-colored graph is called \emph{rainbow~connected} if any two distinct vertices of the graph are connected by a rainbow path in the graph. An edge-coloring under which a graph is rainbow connected is called a \emph{rainbow~connection~coloring} of the graph. Clearly, if a graph is rainbow connected, it must be connected. For a connected graph $G$, the \emph{rainbow connection number} of $G$, denoted by rc$(G)$, is the smallest number of colors that are needed in order to make $G$ rainbow connected. The concept of rainbow connection was introduced by Chartrand et al. \cite{CJMZ} in $2008$. For more details on the rainbow connections, we refer the reader to a book \cite{LS1} and two survey papers \cite{LSS, LS2}.

In this paper, we investigate a new concept introduced by Chartrand et al. in \cite{CDHHZ} that is somehow reverse to the rainbow connection.

An \emph{edge-cut} of a connected graph $G$ is a set $F$ of edges such that $G-F$ is disconnected. The minimum number of edges in an edge-cut of $G$ is the \emph{edge-connectivity} of $G$, denoted by $\lambda(G)$. We have the well-known inequality $\lambda(G)\leq \delta(G)$. For two vertices $u$ and $v$ of $G$, let $\lambda_G(u,v)$ (or simply $\lambda(u,v)$ when the graph $G$ is clear from the context), denote the minimum number of edges in an edge-cut $F$ such that $u$ and $v$ lie in different components of $G-F$. A \emph{u-v-path} is a path with ends $u$ and $v$. The following proposition presents an alternate interpretation of $\lambda(u,v)$ (see \cite{EFS, FF}).

\begin{pro} \cite{EFS, FF}
\emph{For every two vertices $u$ and $v$ in a graph $G$, \emph{$\lambda(u,v)$} is equal to the maximum number of pairwise
edge-disjoint $u$-$v$-paths in $G$}.
\end{pro}

An edge-cut $R$ of an edge-colored connected graph $G$ is called a \emph{rainbow-cut} if no two edges in $R$ are colored the same. A rainbow-cut $R$ of $G$ is said to \emph{separate two distinct
vertices $u$ and $v$} of $G$ if $u$ and $v$ belong to different components of $G-R$. Such a rainbow-cut is called a $u$-$v$-\emph{rainbow-cut}. An edge-colored graph $G$ is called \emph{rainbow~disconnected} if for every two vertices $u$ and $v$ of $G$, there exists a $u$-$v$-rainbow-cut in $G$ separating them. In this case, the edge-coloring is called a \emph{rainbow disconnection coloring} of $G$. For a connected graph $G$, we similarly define
the \emph{rainbow disconnection number} (or rd-\emph{number} for short)
of $G$, denoted by rd$(G)$, as the smallest number of colors that are
needed in order to make $G$ rainbow disconnected. A rainbow disconnection coloring with rd$(G)$ colors is called an rd-\emph{coloring} of $G$.

In \cite{BCLLW}, we introduce the concept of rainbow vertex-disconnection number. For a connected and vertex-colored graph $G$, let $x$ and $y$ be two vertices of $G$. If $x$ and $y$ are nonadjacent, then an $x$-$y$-\emph{vertex-cut} is a subset $S$ of $V(G)$ such that $x$ and $y$ belong to different components of $G-S$. If $x$ and $y$ are adjacent, then an $x$-$y$-\emph{vertex-cut} is a subset $S$ of $V(G)$ such that $x$ and $y$ belong to different components of $(G-xy)-S$. A vertex subset $S$ of $G$ is \emph{rainbow} if no two vertices of $S$ have the same color.
An $x$-$y$-\emph{rainbow-vertex-cut} is an $x$-$y$-vertex-cut $S$ such that if $x$ and $y$ are nonadjacent, then $S$ is rainbow; if $x$ and $y$ are adjacent, then $S+x$ or $S+y$ is rainbow.

A vertex-colored graph $G$ is called \emph{rainbow vertex-disconnected} if for any two vertices $x$ and $y$ of $G$, there exists an $x$-$y$-rainbow-vertex-cut. In this case, the vertex-coloring $c$ is called a \emph{rainbow vertex-disconnection coloring} of $G$. For a connected graph $G$, the \emph{rainbow vertex-disconnection number}
of $G$, denoted by rvd$(G)$, is the minimum number of colors that are needed to make $G$ rainbow vertex-disconnected. A rainbow vertex-disconnection coloring with rvd$(G)$ colors is called an
rvd-\emph{coloring} of $G$.

This paper is organized as follows. In Section $2$, we obtain some upper bounds for rd$(G)$, and moreover, we completely characterize the graphs which meet
the upper bound of the Nordhaus-Gaddum type results obtained early by us. In Section $3$, we propose a conjecture that rd$(G)\leq \lambda^+(G)+1$ and prove it for many classes of graphs, to support it, and moreover, we give a sufficient and necessary condition for a $k$-edge-connected $k$-regular graph $G$ ($k$ is odd) to have rd$(G)=k$. Finally, we give the relationship between rd$(G)$ of $G$ and rvd$(L(G))$ of the line graph $L(G)$ of $G$.

\section{Some upper bounds for rd$(G)$}

In this section, we obtain some upper bounds for the rainbow disconnection number. Let $G$ be a graph and $X$ a proper subset of $V(G)$. To \emph{shrink} $X$ is to delete all the edges between vertices of $X$ and then identify the vertices of $X$ into a single vertex. We denote the resulting graph by $G/X$.
For each vertex $x$ of $G$, let $E_x$ be all edges which are incident with $x$ in $G$.
Now we give some upper bounds for rd$(G)$ in terms of the upper edge-connectivity. First, we give some useful lemmas and introduce a shrinking operation.

\begin{lem}\upshape\cite{CDHHZ}\label{rd-lem1}
If $G$ is a nontrivial connected graph, then
$$\lambda(G) \leq \lambda^+(G)\leq \textnormal{rd}(G)\leq \chi'(G) \leq \Delta(G)+1,$$
where the upper edge-connectivity $\lambda^+(G)$ is defined
by $\lambda^+(G) = \max\{\lambda(u,v): u, v\in V(G)\}.$
\end{lem}

\begin{lem}{\upshape\cite{CDHHZ}}\label{P}
\textnormal{(i)} If $G$ is a Petersen graph, then $\textnormal{rd}(G)$ is 4.

\textnormal{(ii)} If $W_n=C_{n-1}\vee K_1$ is the wheel of order $n\geq 4$, then $\textnormal{rd}(W_n)=3$.
\end{lem}

\begin{lem}{\upshape\cite{BCHL}} \label{Class}
For a graph $G$, the following results hold.

\textnormal{(i)} For any vertex $u$ of $G$, let $H = G-u$. Then $\textnormal{rd}(G)\leq \Delta(H)+1$.

\textnormal{(ii)} If there exists a vertex $u$ of $G$ such that $H=G-u$ is in Class 1 and $d_H(x)\leq \Delta(H)-1$ for any $x \in N_G(u)$, then $\textnormal{rd}(G)\leq \Delta(H)$.
\end{lem}

\noindent\textbf{Remark 1.}
From the proof of Lemma \ref{Class} (i), we know that there exists a rainbow disconnection coloring of $G$ using colors from [$\Delta(H)$+1] such that each vertex is proper except vertex $u$.

\begin{lem}\upshape{\cite{S}}\label{rd-2-multi}
Let $G$ be a loopless multigraph with maximum degree $\Delta(G)$. Then $\chi'(G)\leq \lfloor\frac{3}{2}\Delta(G)\rfloor$.
\end{lem}

\begin{lem}\label{rd-shrink}
Let $G$ be a graph and $H$ a graph by shrinking a vertex subset of $G$ to a single vertex $h$. If $C_H(u,v)$ is a $u$-$v$-edge-cut in $H$, where $u,v\in V(H)\setminus h$, then it is also a $u$-$v$-edge-cut in $G$.
\end{lem}

\begin{proof}
Let $H=G/Y$, where $Y\subseteq V(G)$. Assume that $C_H(u,v)$ is not a $u$-$v$-edge-cut in $G$, namely, there exists a $u$-$v$-path $P$ avoiding $C_H(u,v)$ in $G$. Then the subgraph $P/Y$ of $G/Y$ would also contain a $u$-$v$-path in $G/Y$ avoiding $C_H(u,v)$, a contradiction.
\end{proof}

We define a shrinking operation on a graph $G$ as follows.

For a given graph $G$, let $\lambda^+(G)=k$ and $S=\{x|d(x)\geq k+1\}$.  For fixed $k$ and $S$, suppose $|S|\geq 2$. Let $u,v$ be two vertices of $S$. Then we can find a minimum $u$-$v$-edge-cut $C(u,v)$
such that $|C(u,v)|\leq \lambda^+(G)$ and $G\setminus C(u,v)=C_1\cup C_2$. Then we define the two operations $o$ and $O$ as follows:
$$o(\{G\})=
\begin{cases}
\{G/ V(C_1), G/ V(C_2)\},& \text {if $|G\cap S|\geq 2$},\\
\{G\},& \text{otherwise.}
\end{cases}
$$
$$
O(\{G_1,G_2,\cdots,G_p\})=\cup^p_{i=1}o(\{G_i\}).
$$
We keep the multiple edges in each operation.
Since the graph is split into two pieces when we do
the operation, the operation cannot last endlessly.
Hence, there exists an integer $r$ such that $O^r(\{G\})=O^{r+1}(\{G\})$.
Finally, we get a finite set of connected graphs in which each graph has
at most one vertex with degree at least $\lambda^+(G)+1$. We call this procedure
of splitting and shrinking a graph $G$ into such pieces simply the {\it shrinking operation}
on $G$.

Then we derive the following theorem by the shrinking operation and Lemmas \ref{rd-2-multi} and \ref{rd-shrink}.

\begin{thm}\label{good bound}
Let $G$ be a loopless multigraph with upper edge-connectivity $\lambda^+(G)$. Then $\textnormal{rd}(G)\leq \lfloor\frac{3}{2}\lambda^+(G)\rfloor$.
Moreover, the bound is sharp.
\end{thm}

\begin{proof}
Suppose that we get a family of graph $\mathcal {H}=\{H_1,H_2,\ldots,H_t\}$ by the shrinking operation on $G$.
Obviously, $\Delta(H_i)\leq \lambda^+(G)$ except one vertex of $H_i$ for each $i\in[t]$. For each graph $H_i$, we define a rainbow disconnection coloring $f_i$ as follows.
Let $h_i$ be the unique vertex of $H_i$ with $d_{H_i}(h_i)\geq\lambda^+(G)+1$ and $H_i'=H_i-h_i$.
It follows from Lemmas \ref{rd-lem1} and \ref{rd-2-multi} that
$\chi'(H'_i)\leq \frac{3}{2}\Delta(H_i')\leq \lfloor\frac{3}{2}\lambda^+(G)\rfloor$
for each $i\in[t]$.
For each vertex $u\in N_{H_i}(h_i)$,
since $deg_{H_i}{u} \leq \Delta(H_i)$, there is $a_u\in [\lfloor\frac{3}{2}\lambda^+(G)\rfloor]$ such that the color
$a_u$ is not assigned to any edge incident with $u$.
Define $f_i(h_iu) = a_u$.
Let $w$ and $z$ be two distinct vertices of $H_i$.
Then at least one of the vertices $w$ and $z$ belongs to $H'_i$, say $w\in V(H'_i)$.
Since $E_w$ separates $w$ and $z$ and is rainbow,
it follows that $f_i$ is a rainbow disconnection coloring of $H_i$ using colors from $[\lfloor\frac{3}{2}\lambda^+(G)\rfloor]$  in which each vertex of $H_i$ is proper except vertex $h_i$, where and in what follows a vertex $v$ is \emph{proper} if the set of edges incident with $v$ is rainbow.
Namely, rd$(H_i)\leq \lfloor\frac{3}{2}\lambda^+(G)\rfloor$ for each $i\in[t]$.

Now we claim that we can get a rainbow disconnection coloring of $G$ using colors from $[\lfloor\frac{3}{2}\lambda^+(G)\rfloor]$ by adjusting coloring of shrinking graphs.
Suppose that $F_1$ and $F_2$ are obtained from $F$
by one shrinking operation for vertices $x_1,x_2$ of $F$,
where $d(x_i)\geq \lambda^+(G)+1$ in $F$.
Moreover, suppose that $F_1$ and $F_2$ have a
rainbow disconnection coloring using colors from $[\lfloor\frac{3}{2}\lambda^+(G)\rfloor]$, respectively.
With loss of generality, let $F_i=F/V(F_i)$ and
$x_i\in F_i$ ($i\in[2]$). Let $y_{i}$ be the
vertex by shrinking vertex-set $V(F_i)$ in $F$ ($i\in[2]$).
Note that $d(y_{i})\leq \lambda^+(G)$ in $F_i$ ($i\in[2]$),
so $y_{i}\neq x_{i}$ and the vertex $y_{i}$ is
proper in $F_{i}$.
Thus, we can adjust the colors of edges that are
incident with $y_{2}$ in $F_2$ such that $c(e)|_{F_{1}}=c(e)|_{F_{2}}$ for each $e\in C(x_1,x_2)$ in $F$.
Then we obtain a rainbow disconnection coloring of $F$
by identifying edge-set $C(x_1,x_2)$
for $F_{1}$ and $F_{2}$ using colors from $[\lfloor\frac{3}{2}\lambda^+(G)\rfloor]$.
For any two vertices $p,q$ of $F$,
if $p,q$ belong to $V(F_1)$ and $V(F_2)$, respectively,
then $C(x_1,x_2)$ is a $p$-$q$-rainbow-cut in $F$;
if $p,q$ belong to one of $V(F_1)$ and $V(F_2)$, say $F_{1}$,
then there exists a $p$-$q$-rainbow-cut $C_{F_1}(p,q)$ in
$F_1$ that is also a rainbow-cut in
$F$ by Lemma \ref{rd-shrink}.
Repeating the above inverse shrinking procedure, we finally get a rainbow disconnection coloring of $G$ using colors from $[\lfloor\frac{3}{2}\lambda^+(G)\rfloor]$.
Hence, rd$(G)\leq \lfloor\frac{3}{2}\lambda^+(G)\rfloor$.
Moreover, for the Petersen graph $P$,
we have that rd$(P)=4=\lfloor\frac{3}{2}\lambda^+(P)\rfloor$
since $\lambda^+(P)=3$. Thus, the upper bound is sharp in some sense.
\end{proof}

Next we obtain another bound for rd$(G)$.

\begin{thm}
Let $G$ be a graph of order $n$ with maximum degree $\Delta(G)$ and upper edge-connectivity $\lambda^+(G)$. Then $\textnormal{rd}(G)\leq \min\{n+\lambda^+(G)-\Delta(G)-1, \Delta(G)+1\}$. Furthermore, the bound is sharp.
\end{thm}

\begin{proof}
Let $v$ be a vertex with $d(v)=\Delta(G)$ and $S=V(G)\setminus N[v]$. Then there exist at most $\lambda^+(G)$ edges from $x$ to $N(v)$ for $x\in S$ and $d_{N[v]}(x)\leq \lambda^+(G)$ for $x\in N[v]$ by the definition of upper edge-connectivity. Denote $G'=G-v$. Then $d_{G'}(x)\leq \min\{n+\lambda^+(G)-\Delta(G)-2, \Delta(G)\}$ for any vertex $x\in G'$. So, rd$(G)\leq \Delta(G')+1=\min\{n+\lambda^+(G)-\Delta(G)-1, \Delta(G)+1\}$ by Lemma \ref{Class}.
Moreover, if $G=K_{1,n-1}$ or $W_{n}$, then rd$(G)=n+\lambda^+(G)-\Delta(G)-1$ by Lemmas \ref{P}; if $G$ is the Petersen graph, then rd$(G)=4=\Delta(G)+1$ by Lemma \ref{P}. The upper bound is sharp in some sense.
\end{proof}

In the rest of this section,  we always assume that all graphs have at
least four vertices, and that both $G$ and $\overline{G}$ are connected. For
any vertex $u\in V(G)$, let $\bar{u}$ denote the vertex in $\overline{G}$ corresponding to the vertex $u$. We then characterize the graphs which meet the upper bound of the Nordhaus-Gaddum type results obtained early by us. The following several lemmas will be used.

\begin{lem}{\upshape\cite{ACCGZ}}\label{rd-classs}
Let $G$ be a connected graph. If every connected component of $G_\Delta$
is a unicyclic graph or a tree, and $G_\Delta$ is not a disjoint union of cycles, then $G$ is in Class $1$.
\end{lem}

\begin{lem}{\upshape\cite{BCHL}}\label{rd-n-2}
Let $G$ be a connected graph of order $n$. If $\textnormal{rd}(G)\geq n-2$, then $G$ has at least two vertices of degree at least $n-2$.
\end{lem}

\begin{lem}{\upshape\cite{CDHHZ}}\label{rd-lem4}
If $H$ is a connected subgraph of a graph $G$, then
$\textnormal{rd}(H)\leq \textnormal{rd}(G)$.
\end{lem}

\begin{lem}{\upshape\cite{CDHHZ}}\label{rd-block}
Let $G$ be a connected graph, and let $B$ be a block of $G$ such
that $\textnormal{rd}(B)$ is maximum among all the blocks of $G$. Then
$\textnormal{rd}(G)=\textnormal{rd}(B)$.
\end{lem}

\begin{lem}{\upshape\cite{CDHHZ}}\label{rd-lem3}
Let $G$ be a connected graph of order $n\geq 2$. Then
$\textnormal{rd}(G)=n-1$ if and only if $G$ has at least two
vertices of degree $n-1$.
\end{lem}

\begin{lem}\label{rd-cha-n-2}
Let $G$ be a graph with order $n$. Then
$\textnormal{rd}(G)=n-2$ if and only if one of the following conditions holds.

\textnormal{(i)} $G$ has only one vertex of degree $n-1$ and another  vertex of degree $n-2$.

\textnormal{(ii)} $\Delta(G)=n-2$ and there exist two nonadjacent vertices of degree $n-2$.

\textnormal{(iii)} $\Delta(G)=n-2$ and $G$ has an edge connecting any two vertices of degree $n-2$, and $G$ has a vertex $z$ such that $z\notin N(u)\cup N(v)$ for some pair of vertices $u,v$ of degree $n-2$ or two distinct vertices $x,y$ such that $x\in N(u)\setminus N[v]$ and $y\in N(v)\setminus N[u]$ for some pair vertices $u,v$ of degree $n-2$ and $x,y$ belong to a same component of $G[V\setminus \{u,v\}]$.
\end{lem}

\begin{proof}
For any graph satisfying condition (i), (ii) or (iii), we first get that rd$(G)\leq n-2$ by Lemma \ref{rd-lem3}. Furthermore, we find that $\lambda^+(G)\geq n-2$, and so rd$(G)\geq n-2$ by Lemma \ref{rd-lem1}.

If rd$(G)=n-2$, then $G$ has at least
two vertices of degree at least $n-2$ by Lemma \ref{rd-n-2}.
Furthermore, $G$ does not have two vertices of degree $n-1$.
Therefore, in addition to the graphs satisfying condition (i), (ii) or (iii), the remaining graphs with rd$(G)=n-2$ satisfy the following two conditions:

(1) \ $\Delta(G)=n-2$ and $G$ has an edge connecting any two vertices of  degree $n-2$.

(2) \ $G$ has two distinct vertices $x,y$ such that $x\in N(u)\setminus N[v]$ and $y\in N(v)\setminus N[u]$ for any pair vertices $u,v$ of degree $n-2$, and $x,y$ belong to different component of $G[V\setminus \{u,v\}]$.

We will show that the rainbow disconnection numbers of the graphs satisfying conditions (1) and (2) are at most $n-3$.

If $G[V\setminus \{u,v\}]$ has at least three parts or two parts where each part has at least 2 vertices, then $d(a)\leq n-3$ for $a\in V(G)\setminus \{u,v\}$. We claim that if $G[V\setminus \{u,v\}]$ has two parts where one part has exactly one vertex, then $d(a)\leq n-3$ for $a\in V(G)\setminus \{u,v\}$. Assume that there exists a vertex $w$ of $V(G)\setminus \{u,v\}$ with $d(w)=n-2$. Then $w,v$ are two vertices of degree $n-2$, contradicting to the condition. Let $G'=G-v$. Then $d_{G'}(u)=n-3$ and $d_{G'}(a)\leq n-4$ for $a\in V(G')\setminus \{u\}$. Namely, the graph $G'$ is in Class 1 by Lemma \ref{rd-classs} and $d_{G'}(b)\leq n-3$ for $b\in N(v)$. Thus, we have that rd$(G)\leq n-3$ by Lemma \ref{Class}.
\end{proof}

In \cite{BCHL}, we obtained a Nordhaus-Gaddum type bounds for
rd$(G)$, and examples were given to show that
the upper and lower bounds are sharp. However, we are not satisfied
with these examples, since they are special graphs. We restate it as follows.

\begin{lem}{\upshape\cite{BCHL}}\label{rd-ng}
If $G$ is a connected graph such that $\overline{G}$ is also
connected, then $n-2\leq
\textnormal{rd}(G)+\textnormal{rd}(\overline{G})\leq 2n-5$ and
$n-3\leq \textnormal{rd}(G)\cdot \textnormal{rd}(\overline{G})\leq
(n-2)(n-3)$. Furthermore, these bounds are sharp.
\end{lem}

Next we will completely characterize the graphs which meet
the upper bounds in the above Nordhaus-Gaddum type results,
combining Lemma \ref{rd-cha-n-2}.

\begin{thm}
Let $G$ be a graph of order $n$. Then $\textnormal{rd}(G)+\textnormal{rd}(\overline{G})=2n-5$ $($or $\textnormal{rd}(G)\cdot \textnormal{rd}(\overline{G})=
(n-2)(n-3))$ if and only if one of $G$ and $\overline{G}$ satisfies the following three conditions:

\textnormal{(i)} condoition (ii) or (iii) in Lemma \ref{rd-cha-n-2} holds;

\textnormal{(ii)} it has exactly two vertices of degree $n-2$, say $u,v$;

\textnormal{(iii)} it has at least two vertices of degree 2
except $x,y$ or $z$, where $x\in N(u)\setminus N[v]$, $y\in N(v)\setminus N[u]$ and $z\notin N(u)\cup N(v)$.
\end{thm}

\begin{proof}
Without loss of generality, suppose that $G$ satisfies all above three conditions.
Obviously, rd$(G)=n-2$ by Lemma \ref{rd-cha-n-2}.
Since $G$ has at least two vertices of degree 2 except $x,y$ or $z$,
the graph $\overline{G}\setminus \{\bar{u},\bar{v}\}$ is of order $n-2$ and has at least two vertices of degree $n-3$.
So, rd$(\overline{G})=n-3$ by Lemmas \ref{rd-block} and \ref{rd-lem3}.

Conversely, we know that rd$(G)\leq n-1$ for any connected graph $G$.
Thus, for $\textnormal{rd}(G)+\textnormal{rd}(\overline{G})=2n-5$, by symmetry,
it remains to consider that rd$(G)=n-1$, rd$(\overline{G})=n-4$ and rd$(G)=n-2$, rd$(\overline{G})=n-3$.
Since $\overline{G}$ is connected, we only need to consider
that rd$(G)=n-2$, rd$(\overline{G})=n-3$ by Lemma \ref{rd-lem3}.
Similarly, for $\textnormal{rd}(G)\cdot \textnormal{rd}(\overline{G})=(n-2)(n-3)$,
by symmetry, we only need to consider that rd$(G)=n-2$, rd$(\overline{G})=n-3$.
Obviously, $G$ satisfies (ii) or (iii) of Lemma \ref{rd-cha-n-2}.
So, $G$ does not have any vertex of degree 1.
If $G$ has more than 2 vertices with degree $n-2$,
then $\overline{G}$ has at least 3 vertices with degree 1.
Then rd$(\overline{G})\leq n-4$
by Lemmas \ref{rd-lem4} and \ref{rd-block}.
Thus, condition (ii) holds.
Assume that $G$ has at most one vertex of degree $2$.
Then $\overline{G}$ only has at most one vertex of degree at least
$n-3$ since $G$ does not have any vertex of degree 1.
Moreover, $\overline{G}$ has two vertices $\bar{u},\bar{v}$ of degree 1.
Then rd$(\overline{G})\leq n-4$ by Lemmas \ref{rd-block} and \ref{rd-lem3}.
Assume that in any two vertices of degree 2 of $G$,
at least one of them is $x,y$ or $z$.
Since $\overline{G}$ has two vertices
of degree 1 and $\overline{G}\setminus \{\bar{u},\bar{v}\}$
has at most one vertex of degree $n-3$, rd$(\overline{G})\leq n-4$
by Lemmas \ref{rd-block} and \ref{rd-lem3}. Hence, condition (iii) holds.
\end{proof}

\section{Graphs with rd$(G)\leq \lambda^+(G)+1$}

At first, we recall some known results.

\begin{lem}{\upshape\cite{BCHL}}\label{rd-regu}
If $G$ is a connected $k$-regular graph, then $k \leq
\textnormal{rd}(G) \leq k+1$.
\end{lem}

\begin{lem}{\upshape\cite{BCHL}}\label{rd-thm1}
If $G=K_{n_1,n_2,...,n_k}$ is a complete $k$-partite graph of order $n$ where $k\geq 2$ and $n_1\leq n_2\leq \cdots \leq n_k$,
then
$$\textnormal{rd}(K_{n_1,n_2,...,n_k})=
\begin{cases}
n-n_2,& \text{if $n_1=1$},\\
n-n_1,& \text{if $n_1\geq 2$}.
\end{cases}$$
\end{lem}

\begin{lem}{\upshape\cite{CDHHZ}}\label{rd-2}
The rainbow disconnection number of the grid graph $G_{m,n}$ is as follows.

\textnormal{(i)} For all $n\geq 2$, $\textnormal{rd}(G_{1,n})=\textnormal{rd}(P_n) = 1$.

\textnormal{(ii)} For all $n\geq 3$, $\textnormal{rd}(G_{2,n})=2$.

\textnormal{(iii)} For all $n\geq 4$, $\textnormal{rd}(G_{3,n})=3$.

\textnormal{(iv)} For all $4\geq m\geq n$, $\textnormal{rd}(G_{m,n})=4$.
\end{lem}

Observe that rd$(G)\leq \lambda^+(G)+1$ for all connected regular graphs, complete multipartite graphs and grid graphs.
Therefore, we propose the following conjecture.

\begin{conj}\label{rd-conj}
Let $G$ be a connected graph with upper edge-connectivity $\lambda^+(G)$. Then $\lambda^+(G) \leq \textnormal{rd}(G) \leq \lambda^+(G)+1$.
\end{conj}

Obviously, the lower bound is always true by Lemma \ref{rd-lem1}.
Furthermore, we give some classes of graphs that support the upper bound of the conjecture. The following are some useful lemmas which will be used in the sequel.

\begin{lem}{\upshape\cite{CDHHZ}}\label{rd-tree}
Let $G$ be a nontrivial connected graph. Then $\textnormal{rd}(G)=1$ if and only if $G$ is a tree.
\end{lem}

\begin{lem}{\upshape\cite{CDHHZ}}\label{rd-2}
Let $G$ be a nontrivial connected graph. Then $\textnormal{rd}(G)=2$ if and only if each block of $G$ is either $K_2$ or a cycle and at least one block of $G$ is a cycle.
\end{lem}

\begin{lem}{\upshape\cite{M}}\label{rd-lem2}
Let $G$ be a graph of order $n$ $(n\geq k+2 \geq 3)$. If
$|E(G)|>\frac{k+1}{2}(n-1)-\frac{1}{2}\sigma_k(G),$ where
$\sigma_k(G)=\sum\limits_{\mbox{\tiny $\begin{array}{c}
             x\in V(G) \\
             d(x)\le k \end{array}$}}(k-d(x))$, then
$\lambda^+(G)\geq k+1.$
\end{lem}

\begin{lem}\label{D}
Let $G$ be a connected graph with $\lambda^+(G)=\Delta(G)$. Then $\textnormal{rd}(G)\leq \lambda^+(G)+1$.
\end{lem}

\begin{proof}
It is easy to find that rd$(G)\leq \chi'(G)\leq \Delta(G)+1 =\lambda^+(G)+1$.
\end{proof}

For graphs with small maximum degrees we have the following result.

\begin{thm}\label{3}
Let $G$ be a graph with $\Delta(G)\leq 3$. Then $\textnormal{rd}(G)\leq \lambda^+(G)+1$.
\end{thm}

\begin{proof}
Obviously, $\lambda^+(G)\leq 3$.
If $\lambda^+(G)=1$, we get that $G$ is a tree. It follows from Lemma \ref{rd-tree} that rd$(G)=1=\lambda^+(G)$.
If $\lambda^+(G)=2$, $G$ must contain a cycle and any cycle of $G$ does not have a chord. Thus, $G$ is a cactus graph (i.e., each block of $G$ is a cycle or $K_2$ and at least one block of $G$ is a cycle). It follows from Lemma \ref{rd-2} that rd$(G)=2\leq \lambda^+(G)+1$.
If $\lambda^+(G)=3$, we have that rd$(G)\leq \chi'(G)\leq \Delta(G)+1=4=\lambda^+(G)+1$.
\end{proof}

For graphs with large maximum degrees we have the following result.

\begin{thm}\label{n-3}
Let $G$ be a graph with $\Delta(G)\geq n-3$. Then $\textnormal{rd}(G)\leq \lambda^+(G)+1$.
\end{thm}

\begin{proof}
Let $d(u)=\Delta(G)$ and $G'=G-u$. Suppose $\lambda^+(G)=k$.
If $\Delta(G)\geq n-2$, we have $\Delta(G')\leq k$;
otherwise, let $v$ be a vertex with $d_{G'}(v)\geq k+1$. Then we have $\lambda^+(u,v)\geq k+1$, a contradiction.
Thus, rd$(G)\leq \Delta(G')+1\leq k+1$ by Lemma \ref{Class}.

If $\Delta(G)=n-3$, let $d(u)=n-3$ and let
$p,q$ be two vertices which are not adjacent to $u$
(i.e. $V(G)=N[u]\cup\{p,q\}$).
Note that $d_G(x)\leq k+2$ for $x\in N(u)$ and $d_G(p), d_G(q)\leq k+1$ since $\lambda^+(G)=k$. Thus, $\Delta(G')\leq k+1$. We distinguish the following cases to discuss.

\textbf{Case 1}. $\Delta(G')\leq k$.

It follows from Lemma \ref{Class} that rd$(G)\leq \Delta(G')+1=k+1$.

\textbf{Case 2}. $\Delta(G')=k+1$.

Let $D=\{x|d_{G'}(x)=k+1\}$.
If $D\subseteq \{p,q\}$, then $G'_\Delta$ is $K_1$ (otherwise, $\lambda(p,q)=k+1$, a contradiction).
Thus, it follows from Lemma \ref{rd-classs}
that $G'$ is in Class 1.
Moreover, $d_{G'}(x)\leq \Delta(G')-1$ for every $x\in N_G(u)$.
So, rd$(G)\leq \Delta(G')=k+1$ by Lemma \ref{Class}.

Suppose $D\cap N(u)\neq \phi$. We claim that $|D\cap N(u)|=1$.
Assume that there are at least two vertices in $D\cap N(u)$, say $x_1,x_2$. Note that $d_{N(u)}(x_1)=d_{N(u)}(x_2)\leq k-1$. So, $\{p,q\} \subseteq N(x_i)$ for each $i\in [2]$. Then we have $\lambda^+_G(x_1,x_2)\geq k+1$, a contradiction.
Let $D\cap N(u)=\{a\}$. Then $\{p,q\}\subseteq N(a)$. Let
$R=N(u)\setminus N[a]$, $T=N(p)\cup N(q)$.
Note that any vertex of $R$ is not adjacent to $T\cup\{p,q\}$.
Assume that there exists a vertex of $R$ which is adjacent to a vertex of $T\cup\{p,q\}$. Then we have $\lambda^+(u,a)\geq k+1$, a contradiction.
Thus, $T\subseteq N[a]$. Let $S=N[a]\setminus T$.
If there exists a vertex $s\in S$ such that
$s$ belongs to a component with a vertex
of $R$ in $G[R\cup S]$, then let $s\in S_1$ and $S_2=S\setminus S_1$.
Observe that the edge-set $E(u,S_2\cup T)\cup E(S_1,a)$
is a $u$-$a$-edge-cut by the definitions of $R$, $S_1$ and $S_2$.
Let $G_1=G[R\cup S_1\cup \{u,a\}]-ua$ and $G_2=G[T\cup S_2\cup \{u,p,q\}]$. Write $G'_1=G_1-u$ and $G'_2=G_2-a$. Then we have  $\Delta(G'_1),\Delta(G'_2)\leq k$.
By Lemma \ref{Class} and Remark 1, there exists a rainbow disconnection coloring $c_i$ of $G_i$ $(i\in[2])$ using colors from $[k+1]$ , moreover, vertex $x$ is proper for each $x \in V(G_1)\setminus \{u\}$ ($x \in V(G_2)\setminus \{a\}$) in coloring $c_1$ of $G_1$ ($c_2$ of $G_2$).
Since $|E(u,S_2\cup T)\cup E(S_1,a)|=k$,
we can adjust colors of $E(S_1,a)$
such that $E(u,S_2\cup T)\cup E(S_1,a)$
have distinct colors. Then we get a coloring
$c$ of $G$ by identify the graph $G_1$ and
$G_2$ using colors from $[k+1]$.

Furthermore, we can verify that $c$ is a
rainbow disconnection coloring of $G$.
For any two vertices $w,z$ of $G$,
if there exists a vertex not in $\{u,a\}$, say $w$, then $E_w$ is a $w$-$z$-rainbow-cut;
if $\{w,z\}=\{u,a\}$, then $E(u,S_2\cup T)\cup E(S_1,a)$
is a $u$-$a$-rainbow-cut. Hence, rd$(G)\leq k+1$.
\end{proof}

By Theorems \ref{3} and \ref{n-3},
we get the following result for graphs of small orders.
\begin{cor}
Let $G$ be a graph of order $n\leq 7$. Then rd$(G)\leq \lambda^+(G)+1$.
\end{cor}

We recall some notions of graphs from \cite{H}. A simple graph $G$ is \emph{overfull} if $|E(G)| > \lfloor \frac{n}{2}\rfloor \Delta(G)$.
A graph $G$ is \emph{subgraph-overfull} if it has an overfull subgraph $H$ with $\Delta(H)=\Delta(G)$. Obviously, every overfull graph is subgraph-overfull. For dense graphs we have the following result.

\begin{thm}
Let $G$ be a subgraph-overfull graph with order $n$ and upper edge-connectivity $\lambda^+(G)$. Then $\textnormal{rd}(G)\leq \lambda^+(G)+1$.
\end{thm}

\begin{proof}
Let $H$ be an overfull subgraph of $G$ with $\Delta(H)=\Delta(G)$.
Then, $|E(H)|>\lfloor \frac{|V(H)|}{2}\rfloor \Delta(H)\geq \frac{|V(H)|-1}{2} \Delta(H)$. Thus, we have that $\Delta(G)\geq \lambda^+(G)\geq \lambda^+(H)\geq \Delta(H)=\Delta(G)$ by Lemma \ref{rd-lem2}. So, $\lambda^+(G)=\Delta(G)$.
Hence, we have rd$(G)\leq \lambda^+(G)+1$ by Lemma \ref{D}.
\end{proof}
For a $k$-regular graph $G$, it follows from Lemma \ref{rd-regu} that the conjecture is true since $\lambda^+(G)=k$. However, we want further to know the $k$-regular graphs with rd$(G)=k$. In \cite{BCHL}, we presented some results on this kind of graphs. We now deduce the following result for $k$-edge-connected $k$-regular graphs with $k$ being odd.

\begin{thm}
Let $k$ be an odd integer, and $G$ a $k$-edge-connected $k$-regular graph of order $n$. Then $\chi'(G)=k$ if and only if $\textnormal{rd}(G)=k$.
\end{thm}

\begin{proof}
Suppose, first, that $\chi'(G)=k$. By Lemma \ref{rd-lem1}, we have that $k=\lambda(G)\leq \textnormal{rd}(G)\leq \chi'(G)=k$. Thus, $\textnormal{rd}(G)=k$.

Conversely, suppose that $\textnormal{rd}(G)=k$ and
let $c$ be an rd-coloring of $G$.
If $G$ has a $k$-rainbow-cut $T$ such that
$G\setminus T$ has two non-trivial components,
say $G_1$, $G_2$, then we do an operation $f$,
i.e., the graph $G$ shrinks $V(G_1)$, $V(G_2)$,
respectively, to vertices $x_1, x_2$.
The resulting edge-colored graphs are denoted
by $G/V(G_1)$, $G/V(G_2)$, respectively.
Furthermore, the obtained edge-colored graphs
$G/V(G_1)$ and $G/V(G_2)$ are both $k$-edge-connected
$k$-regular. Assume, without loss of generality,
that there exists a $u$-$v$-edge-cut $V$ in $G/V(G_1)$,
where $u,v \in G/V(G_1)$ and $|V|<k$.
By Lemma \ref{rd-shrink},
we know that $V$ is also a $u$-$v$-edge-cut in $G$, a contradiction.

\textbf{Claim 1.} The coloring $c$ of $G$ restrict to $G/V(G_1)$ is an rd-coloring of $G/V(G_1)$.

\noindent\textbf{Proof of Claim 1:}
Note that $V(G/V(G_1))=V(G_2)\cup \{x_1\}$.
Let $u,v$ be two vertices of $G/V(G_1)$.
Suppose $u,v \in V(G_2)$. Let $W$ be a minimum $u$-$v$-rainbow-cut in $G$ and let $W_{H}$ be the set of edges in $W\cap H$.
Since $G_1$, $G_2$ are both $\lceil\frac{k}{2}\rceil$-connected,
we have $|W_{G_2}|\geq \lceil\frac{k}{2}\rceil$.
If the remaining edges of $W$ are all in $G_1$,
then there still is a $u$-$v$-path in $G\setminus W$ since $G_1$ is $\lceil\frac{k}{2}\rceil$-connected and
$|W_{G_1}|\leq \lfloor\frac{k}{2}\rfloor< \lceil\frac{k}{2}\rceil$ for $k$ odd, a contradiction.
If $G_1$ and $T$ both have edges in $W$, without loss of generality,
suppose that $|W_{G_1}|=s$, $|W_{T}|=t$ and $|W_{G_2}|=r$,
where $0<t,s <\lfloor\frac{k}{2}\rfloor$,
$s+t \leq \lfloor\frac{k}{2}\rfloor$ and $r+s+t=k$.
When we remove the set $W$ from $G$, at most $t$ $u$-$v$-paths that go through $T$ are destroyed.
However, there are $s+t$ $u$-$v$-paths going through $T$ in $G$, and
so at least one $u$-$v$-path goes through $T\setminus W$ since $s\geq 1$. Moreover, $G_1\setminus W$ is connected since $|W_{G_1}|< \lceil\frac{k}{2}\rceil$ and $G_1$ is $\lceil\frac{k}{2}\rceil$-connected.
So, there is at least one $u$-$v$-path in $G\setminus W$, a contradiction. Hence, $W\subseteq G_2\cup T$.
Then $W$ is a $u$-$v$-rainbow-cut of $G/V(G_1)$
(otherwise, if $G/V(G_1)$ has a $u$-$v$-path
avoiding the set $W$, then there exists a $u$-$v$-path in $G\setminus W$, a contradiction). If one of $u,v$ is $x_1$, say $u=x_1$,
then $E_{x_1}$ is a $u$-$v$-rainbow-cut of $G/V(G_1)$.
\qed

Repeating the operation $f$ until the obtained edge-colored graphs do not satisfy the condition of operation $f$, the resulting edge-colored $k$-edge-connected $k$-regular graphs are denoted by $\mathscr{F}=\{F_i|i\in [\ell]\}$.

\textbf{Claim 2.} The coloring of the graph $F_i$ in $\mathscr{F}$ is a proper coloring of $F_i$ for each $i\in[\ell]$.

\noindent\textbf{Proof of Claim 2:}
Assume that there exists a graph $F_i$
for some $i\in [\ell]$ for which the coloring is not proper.
If $F_i$ has two vertices, say $p,q$, which are not proper,
then there exists a $p$-$q$-rainbow-cut $Z$ in $F_i$
that are not $E_p$ or $E_q$.
Thus, we get that $Z$ is a rainbow-cut in $F_i$ such
that $F_i\setminus Z$ has two non-trivial components,
a contradiction with the operation $f$.
Hence, $F_i$ has at most one vertex, say $b_i$, which
is not proper for each $i\in[\ell]$.
Given an $i\in[\ell]$, let $k_t$ $(t\in[k])$ be the number
of edges incident with vertex $b_i$ and with color $t$ in $F_i$,
and moreover,
let $F_{i,A_{j}}$ be an induced subgraph of
$F_i$ by the set of edges with colors in $A_{j}$,
where $A_{j}$ is the color set $[k]\setminus\{j\}$ for $j\in [k]$.
Then for the graph $F_i$ ($i\in [\ell]$), $(k-1)(|F_i|-1)+\sum_{t\in A_{j}}{k_t}\equiv 0\pmod 2$ since the sum of degrees of vertices in $F_{i,A_{j}}$ is even for each $j \in [k]$. Furthermore, we have that
$\sum_{t\in A_{1}}{k_t}\equiv \sum_{t\in A_{2}}{k_t}\equiv \cdots \equiv\sum_{t\in A_{k}}{k_t}\equiv0\pmod 2$,
namely, $k_1\equiv k_2 \equiv\cdots \equiv k_k \equiv0\pmod 2$.
Combined with $\sum_{i=1}^k k_i=k$, we obtain that $k_1=k_2=\cdots=k_k=1$. So, the vertex $b_i$ is also proper in $F_i$ for each $i\in [\ell]$.
\qed

For each vertex $x$ of $G$, the colors of edges incident with vertex $x$ are not change in each operation $f$. Thus, the rd-coloring $c$ of $G$ is a proper coloring of $G$, i.e., rd$(G)\geq \chi'(G)$. Hence, $\chi'(G)=k$.
\end{proof}

\section{Relationship of rd$(G)$ and rvd$(L(G))$}

The \emph{line graph} $L(G)$ of a graph $G$ has the edges of $G$ as its vertices, and two distinct edges of $G$ are adjacent in $L(G)$ if and only if they share a common vertex in $G$.
Now, we study the relationship between rd$(G)$ and rvd$(L(G))$.

\begin{lem}\upshape{\cite{BCLLW}}\label{rvdcomplete}
For an integer $n\geq 2$, $$\textnormal{rvd}(K_{n})=\left\{
\begin{array}{lcl}
n-1,       &      & {if~n=2,3},\\
n,         &      & {if~n\geq 4}.
\end{array} \right .$$
\end{lem}

\begin{thm}\label{rd-line}
Let $G$ be a graph and $L(G)$ the line graph of $G$. Then $\textnormal{rd}(G)\leq \textnormal{rvd}(L(G))$.
\end{thm}

\begin{proof}
Let $c_0$ be an rvd-coloring of the line graph $L(G)$.
Then we get an edge-coloring $c$ of $G$ since the edge-colorings of $G$ are one-to-one correspondence with the vertex-colorings of $L(G)$.
We can verify that $c$ is a rainbow disconnection coloring of $G$.
For any two vertices $u,v$ of $G$, if $uv$ is not a pendent edge, we can find two edges $e_1, e_2$ incident with vertices $u,v$, respectively, and the edge $e_1$ (or $e_2$) does not have two ends as $u,v$. Suppose that $e_1=ux$ and $e_2=vy$, where $x,y\in V(G)\setminus\{u,v\}$ and $x,y$ could be the same vertex.
We know that $e_1,e_2$ correspond to two vertices of $L(G)$, denoted by $a$ and $b$.
We claim that the edge-set $S$ of $G$ which
corresponds to an $a$-$b$-rainbow-vertex-cut $S'$ in $L(G)$ is a $u$-$v$-rainbow-cut in $G$.
Assume that there still exists a $u$-$v$-path $P$ in $G$ which avoids the edge-set $S$ of $G$. Then the $u$-$v$-path $P$ in $G$ corresponds to an $a$-$b$-path $P'$ which avoids the vertex-set $S'$ in $L(G)$. A contradiction.
If $uv$ is a pendent edge of $G$, then $uv$ is a $u$-$v$-rainbow-cut in $G$.
\end{proof}

It is easy to know that the chromatic index of $G$ is equal to the chromatic number of $L(G)$. However, we can only have rd$(G)\leq \textnormal{rvd}(L(G))$ from Theorem \ref{rd-line}. The equality
is not always true. For the moment we have the following necessary
condition for the equality.

\begin{thm}
Let $G$ be a graph with $\delta(G)\geq 4$ and $L(G)$ the line graph of $G$. If $\textnormal{rd}(G)=\textnormal{rvd}(L(G))$, then $\textnormal{rd}(G)=\chi'(G)$.
\end{thm}

\begin{proof}
By contradiction, assume that $t=\textnormal{rd}(G)=\textnormal{rvd}(L(G))<\chi'(G)$.
Let $c$ be a coloring of $G$ using colors from $[t]$.
Then there exists at least one vertex, say $v$, such
that $E_v$ has at least two edges with the same color.
Since $\delta(G)\geq 4$, $E_v$ in $G$ corresponds to a $K_t$ in $L(G)$,
where $t=|N(v)|\geq 4$.
Note that there are at most $t-1$ colors in $K_t$.
This is a contradiction to Lemma \ref{rvdcomplete}.
\end{proof}


\begin{thebibliography}{1}

\bibitem{ACCGZ}
S. Akbari, D. Cariolaro, M. Chavooshi, M. Ghanbari, S. Zare,
Some criteria for a graph to be in Class $1$,
{\it Discrete Math.}
\textbf{312}(2012), 2593--2598.

\bibitem{BCHL}
X. Bai, R. Chang, Z. Huang, X. Li,
More on rainbow disconnection in graphs,
arXiv:1810.09736 [math.CO].

\bibitem{BCLLW}
X. Bai, Y. Chen, X. Li, P. Li, Y. Weng,
The rainbow vertex-disconnection in graphs,
arXiv:1812.10034 [math.CO].

\bibitem{BM}
J.A. Bondy, U.S.R. Murty,
Graph Theory,
Graduate Texts in Mathematics 244, Springer, 2008.

\bibitem{CDHHZ}
G. Chartrand, S. Devereaux, T.W. Haynes, S.T. Hedetniemi, P. Zhang,
Rainbow disconnection in graphs,
{\it Discuss. Math. Graph Theory}
\textbf{38}(2018), 1007--1021.

\bibitem{CJMZ}
G. Chartrand, G.L. Johns, K.A. McKeon, P. Zhang,
Rainbow connection in graphs,
{\it Math. Bohem.}
\textbf{133}(2008), 85--98.

\bibitem{EFS}
P. Elias, A. Feinstein, C.E. Shannon,
A note on the maximum flow through a network,
{\it IRE Trans. Inform. Theory, IT}
\textbf{2}(1956), 117--119.

\bibitem{FF}
L.R. Ford Jr.,  D.R. Fulkerson,
Maximal flow through a network,
{\it Canad. J. Math.}
\textbf{8}(1956), 399--404.

\bibitem{H}
A.J.W. Hilton, Two conjectures on edge-colouring,
{\it Discrete Math.}
\textbf{74}(1989), 61--64.

\bibitem{LSS}
X. Li, Y. Shi,  Y. Sun,
Rainbow connections of graphs: A survey,
{\it Graphs  Combin.}
\textbf{29}(2013), 1--38.

\bibitem{LS1}
X. Li, Y. Sun,
Rainbow Connections of Graphs,
Springer Briefs in Math.,
Springer, New York, 2012.

\bibitem{LS2}
X. Li, Y. Sun,
An updated survey on rainbow connections of graphs - a dynamic survey,
{\it Theo. Appl. Graphs.}
\textbf{0}(2017), Art. 3, 1--67.

\bibitem{M}
W. Mader, Ein extremalproblem des zusammenhangs von graphen,
{\it Math. Z.}
\textbf{131}(1973), 223--231.

\bibitem{S}
C.E. Shannon,
A theorem on coloring the lines of a network,
{\it Math. Phys.}
\textbf{28}(1949), 148--152.

\bibitem{V}
V.G. Vizing,
On an estimate of the chromatic class of a $p$-graph,
{\it Diskret. Anal.}
\textbf{3}(1964), 25--30, in Russian.

\end{thebibliography}
\end{document}